\documentclass[12pt]{amsart}
\usepackage{amscd, amssymb, amsmath, amsthm, amsfonts}
\usepackage{mathrsfs}
\usepackage{amscd, amssymb, amsmath, amsthm, amsfonts}
\usepackage{latexsym, graphics, graphicx, psfrag}
\usepackage{color,xcolor,colortbl}
\usepackage[all]{xy}
\usepackage[small,nohug,heads=vee]{diagrams}
\diagramstyle[labelstyle=\scriptstyle]

\newtheorem{theorem}{Theorem}[section]
\newtheorem{lemma}[theorem]{Lemma}
\newtheorem{proposition}[theorem]{Proposition}
\newtheorem{corollary}[theorem]{Corollary}

\theoremstyle{definition}

\newtheorem{example}[theorem]{Example}
\newtheorem{question}[theorem]{Question}
\newtheorem{remark}[theorem]{Remark}

\begin{document}


\title[Maps from 3-manifolds to 4-manifolds]
{Maps from 3-manifolds to 4-manifolds \\ that induce isomorphisms on $\pi_1$}

\author{Hongbin Sun}
\address{Department of Mathematics, Rutgers University - New Brunswick, Hill Center, Busch Campus, Piscataway, NJ 08854, USA}
\email{hongbin.sun@rutgers.edu}

\author{Zhongzi Wang}
\address{Department of Mathematics Science, Tsinghua University, Beijing, 100080, CHINA}
\email{wangzz18@mails.tsinghua.edu.cn}




\thanks{The first author is partially supported by Simons Collaboration Grant 615229.}

\begin{abstract}
In this paper, we prove that if a closed orientable $3$-manifold $M$ admits a map $f:M\to N$ to a closed orientable $4$-manifold $N$ such that $f$ induces an isomorphism on fundamental groups, then
$M$ is homeomorphic to $\#^k S^1\times S^2$ or $S^3$.
Relevant results on higher dimensional manifolds are also obtained.
\end{abstract}

\date{} 
\maketitle
\tableofcontents
\section{Introduction}
All manifolds in this note are compact, connected and orientable unless otherwise stated.
We first state the following known result.

\begin{theorem}\label{basic}
\begin{enumerate}
\item For $n=1,2$, there is no closed $n$-manifold $M$ that bounds an $(n+1)$-manifold $N$ such that 
the inclusion $i: M\to N$ induces an isomorphism on their fundamental groups, unless $M$ is the 2-sphere.

\item For any $n\ge 4$ and any finitely presented group $G$, 
there is a closed $n$-manifold  $M$ that bounds an $(n+1)$-manifold $N$ such that $\pi_1(M)\cong G$ and
the inclusion $i: M\to N$ induces an isomorphism on their fundamental groups.

\item For $n=3$, there is no closed $3$-manifold $M$ that bounds a $4$-manifold $N$ such that the inclusion $i: M\to N$ induces an isomorphism on their fundamental groups, 
unless $M$ is the 3-sphere or a connected sum of $S^1\times S^2$.
\end{enumerate}
\end{theorem}

For Theorem \ref{basic} (1), the $n=1$ case follows from the classification of 1-manifolds and 2-manifolds, and the $n=2$ case follows from a classical result on 3-manifolds
\cite[Theorem 10.2]{He}. 
Theorem \ref{basic} (2)
follows  from the famous construction  that any finitely presented group $G$ can be 
realized as the fundamental group of some closed $n$-manifold 
for any $n\ge 4$.
Theorem \ref{basic} (3) is proved by Daverman in \cite{Da}.

The main result of this paper is the following theorem.

\begin{theorem}\label{3-4map}
Let $M$ be a closed orientable $3$-manifold. There exists a map $f:M\to N$ to a closed orientable $4$-manifold $N$ such that $f_*:\pi_1(M)\to \pi_1(N)$ is an isomorphism if and only if $M$ is homeomorphic to $\#^k S^1\times S^2$ for some $k\in \mathbb{Z}_{\geq 0}$.
\end{theorem}

Theorem \ref{3-4map} is  related to Theorem \ref{basic} (3).  If $M$ is a closed connected orientable $3$-manifold that bounds a compact orientable $4$-manifold $N$ such that the inclusion induces an isomorphism on $\pi_1$, then the natural embedding from $M$ to the double of $N$ induces an isomorphism on $\pi_1$.

Since $\#^k S^1\times S^2$ bounds $\natural^k S^1\times D^3$ and the inclusion induces an isomorphism on $\pi_1$, there do exist maps from $M^3=\#^k S^1\times S^2$ to $\#^kS^1\times S^3$ that induce isomorphisms on $\pi_1$, via the doubling construction. 


The proof of Theorem \ref{3-4map} divides into the following two cases, and in each case we prove that a map $f:M\to N$ that induces an isomorphism on $\pi_1$ does not exist.
\begin{enumerate}
\item $M=U\#V$ such that $U$ is an aspherical $3$-manifold.
\item $M=(\#_{j=1}^n N_j)\# (\#^k S^1\times S^2)$ with $n\geq 1$, such that $|\pi_1(N_j)|$ is finite and nontrivial for all $j$.
\end{enumerate}
These two cases will be proved in Propositions \ref{aspherical} and \ref{2powerandZ} respectively. 

Note that in Theorem \ref{3-4map}, we cannot drop the orientability condition on the $4$-manifold $N$. For example, the natural embedding from $RP^3$ to $RP^4$ induces an isomorphism on $\pi_1$.

Here is an interesting consequence of the proof of the first case of Theorem \ref{3-4map}.


\begin{corollary}\label{3-3}
Let $M$ be a closed orientable $3$-manifold whose fundamental group is not free, and let $f:M\to M$ be a map that induces an isomorphism on the fundamental group. Then $f$ is a non-zero degree map.
\end{corollary}

Corollary \ref{3-3} was known for prime 3-manifolds. It is a classical fact when $M$ is aspherical, and it is also verified for 
3-manifolds with finite fundamental groups in some papers, say \cite{Ol} and \cite{HKWZ}.




Proposition \ref{aspherical} (on $3$-manifolds with aspherical prime factors) actually
holds for any dimension $n$ if we replace ''map" by "embedding''. For prime decomposition of $n$-manifolds, see \cite{BCFHKLN}.  

\begin{theorem}\label{n}
Suppose $N^{n+1}$ is a compact orientable manifold of dimension $n+1$ such that $\partial N$ is either empty or connected,
and $M^n$ is a closed orientable manifold of dimension $n$.\\
If $M^n$ contains an aspherical prime factor, then there is no embedding  $i: M \to N$ that induces
an isomorphism $i_* : \pi_1(M) \to \pi_1(N)$.\\
In particular, $M^n$ does not bound a compact orientable $(n+1)$-manifold such that the inclusion induces an isomorphism on $\pi_1$.
\end{theorem}

In Section 2, after reviewing basic results on 3-manifolds, we  study $\mathbb{Z}$-coefficient (co)homology of $3$-manifold groups, then use it to prove the first case of Theorem \ref{3-4map} (Proposition \ref{aspherical}). In Section 3, we use $\mathbb{Z}_p$-coefficient (co)homology of $3$-manifold groups to prove the second case of Theorem \ref{3-4map} (Proposition \ref{2powerandZ}).
The proof of Theorem   \ref{n}  will be presented in Section 4. 
Although maps from 3-manifolds to 4-manifolds realizing isomorphisms on $\pi_1$ are rare (Theorem \ref{3-4map}),
there are plenty of maps from 4-manifolds to 3-manifolds realizing  isomorphisms on $\pi_1$, as constructed in Section 5.

{\bf Acknowledgements.}
Theorem 1.2 is partly inspired by communications with  Yuguang Shi, Jiajun Wang and Shicheng Wang. 
 Yi Ni and Yang Su independently also made some progress toward Theorem 1.2. The authors thank the referee for many very helpful suggestions on simplifying the structure and the proofs of this paper.

\section{Maps from 3-manifolds to 4-manifolds and $\mathbb{Z}$-coefficient (co)homology of $3$-manifold groups}

At first, we need the following classical result on 3-manifold theory.

\begin{theorem}\label{known}
Each closed orientable 3-manifold $M$ other than $S^3$ has a prime decomposition 
$$M=(\#_{i=1}^mM_i)\#(\#_{j=1}^n N_j)\#(\#^kS^1\times S^2).$$ 
Here each $M_i$ is aspherical, each $N_j$ has finite and nontrivial fundamental group, indeed has $S^3$ as its universal cover. 
\end{theorem}

The proof of this result follows from Kneser-Milnor's prime decomposition theorem and the sphere theorem (\cite[Chap. 3 and 4]{He}), the Whitehead theorem and the Hurewicz theorem in algebraic topology, and Perelman's confirmation of Thurston's geometrization of $3$-manifolds.





Most of the proofs in this section are based on the following result on $\mathbb{Z}$-coefficient (co)homology of $3$-manifold groups.

\begin{proposition}\label{kpi1}
Let $M$ be a closed orientable $3$-manifold whose fundamental group is not free. Let $\pi=\pi_1(M)$, let $i:M\to K(\pi,1)$ be the unique map (up to homotopy) that induces the identity homomorphism on fundamental group. Then the following statements hold.

(1) The homomorphism $$i_*:H_3(M;\mathbb{Z})\to H_3(K(\pi,1);\mathbb{Z})$$ is not trivial. Moreover, if $M$ has an aspherical prime summand, $i_*:H_3(M;\mathbb{Z})\to H_3(K(\pi,1);\mathbb{Z})$ is injective.

(2) $H_4(K(\pi,1);\mathbb{Z})=0$.
\end{proposition}

Here we only sketch the proof since it is well-known for experts. Suppose $M$ is expressed as in Theorem \ref{known}, then a model of $K(\pi,1)$ is 
$$(\vee_{i=1}^m M_i)\vee(\vee_{j=1}^nK(\pi_1(N_j),1))\vee (\vee^k S^1),$$
with $m+n\geq 1$. 
The map $i:M\to K(\pi,1)$ is obtained by pinching each decomposition sphere of $M$ to a point, and projecting each $S^1\times S^2$ prime factor to $S^1$. Proposition \ref{kpi1} (1) obviously holds if $m\geq 1$. If $m=0$, then $n\geq 1$, and Proposition \ref{kpi1} (1)  holds since $H_3(N_j;\mathbb{Z})\cong \mathbb{Z}\to H_3(K(\pi_1(N_j),1);\mathbb{Z})\cong \mathbb{Z}_{|\pi_1(N_j)|}$ is surjective. Proposition \ref{kpi1} (2) holds since $H_4(K(\pi_1(N_j),1);\mathbb{Z})=0$, via a construction of $K(\pi_1(N_j),1)$ by adding one $4$-cell and more cells of dimension at least $5$.

The following result is a useful consequence of Proposition \ref{kpi1}.

\begin{proposition}\label{general}
Let $M$ be a closed orientable $3$-manifold whose fundamental group is not free, and let $X$ be a cell-complex. If $f:M\to X$ is a map that induces an isomorphism on fundamental groups, then the homomorphism $$f_*:H_3(M;\mathbb{Z})\to H_3(X;\mathbb{Z})$$ is not trivial. Moreover, if $M$ has an aspherical prime summand, then $f_*:H_3(M;\mathbb{Z})\to H_3(X;\mathbb{Z})$ is injective.
\end{proposition}

\begin{proof}
We use $\pi$ to denote $\pi_1(M)$. There are maps $i:M\to K(\pi,1)$ and $j:X\to K(\pi,1)$ (up to homotopy) that induce isomorphisms on fundamental groups.

Then there exists $h:K(\pi,1)\to K(\pi,1)$ such that the following diagram commutes up to homotopy
\begin{diagram}
K(\pi,1) &\rTo^h & K(\pi,1)\\
\uTo_{i}& &\uTo_{j}\\
M&\rTo^f &X.
\end{diagram}

Since all of $f$, $i$, $j$ induce isomorphisms on $\pi_1$, $h$ also induces an isomorphism on $\pi_1$,
and $h_*:H_3(K(\pi,1);\mathbb{Z})\to H_3(K(\pi,1);\mathbb{Z})$ is an isomorphism.

Now we have an induced commutative diagram on $H_3$:
\begin{diagram}
H_3(K(\pi,1);\mathbb{Z}) &\rTo^{h_*} & H_3(K(\pi,1);\mathbb{Z})\\
\uTo_{i_*}& &\uTo_{j_*}\\
H_3(M;\mathbb{Z}) &\rTo^{f_*} & H_3(X;\mathbb{Z}).
\end{diagram}

By Proposition \ref{kpi1} (1), we know that $i_*$ is nontrivial. Since $h_*$ is an isomorphism, $h_*\circ i_*=j_*\circ f_*:H_3(M;\mathbb{Z})\to H_3(K(\pi,1);\mathbb{Z})$ is nontrivial. So $f_*$ is nontrivial.

If $M$ has an aspherical prime summand, then Proposition \ref{kpi1} (1) implies that $i_*$ is injective. Then the above commutative diagram implies that $f_*$ is injective.
\end{proof}



\begin{proof}[Proof of Corollary \ref{3-3}]
Let $M$ be a closed orientable $3$-manifold whose fundamental group is not free, and let $f:M\to M$ be a map that induces an isomorphism on the fundamental group. According to Proposition 2.3, $f_{*}: H_3(M,  \mathbb{Z})\to H_3(M,  \mathbb{Z})$ is non-trivial. Note that $H_3(M;\mathbb{Z})\cong \mathbb{Z}$, so $f$ is a non-zero degree map.
\end{proof}

Proposition \ref{general} has one more corollary.

\begin{corollary}\label{4mfld}
Let $M$ be a closed orientable $3$-manifold whose fundamental group is not free, and let $N$ be a closed orientable $4$-manifold. If $f:M\to N$ is a map that induces an isomorphism on fundamental groups, then the homomorphism $$f_*:H_3(M;\mathbb{Z})\to H_3(N;\mathbb{Z})$$ is injective.
\end{corollary}

\begin{proof}
Since $N$ is a closed orientable $4$-manifold, the Poincare duality implies $H_3(N;\mathbb{Z})\cong H^1(N;\mathbb{Z})$ and it is torsion free. Since Proposition \ref{general} implies that $f_*:H_3(M;\mathbb{Z})\to H_3(N;\mathbb{Z})$ is nontrivial and $H_3(M;\mathbb{Z})\cong \mathbb{Z}$, $f_*$ must be injective.
\end{proof}





Now we prove the first case of Theorem \ref{3-4map}.

\begin{proposition}\label{aspherical}
Let $M$ be a closed orientable $3$-manifold with an aspherical $3$-manifold in its prime decomposition, and let $N$ be a closed orientable $4$-manifold. Then there is no map $f:M\to N$ such that $f_*:\pi_1(M)\to \pi_1(N)$ is an isomorphism.
\end{proposition}

\begin{proof}

We suppose that such a map $f:M\to N$ exists. Let $\pi=\pi_1(M)$, then we take a map $k:N\to K(\pi,1)$ such that $$k\circ f:M\xrightarrow{f}N\xrightarrow{k} K(\pi,1)$$ induces the identity homomorphism on $\pi_1$. 

Then we have an induced homomorphism on cohomology rings $$(k\circ f)^*:H^*(K(\pi,1);\mathbb{Z})\xrightarrow{k^*} H^*(N;\mathbb{Z})\xrightarrow{f^*} H^*(M;\mathbb{Z}).$$

By the moreover part of Proposition \ref{kpi1} (1), $(k\circ f)_*:H_3(M;\mathbb{Z})\to H_3(K(\pi,1);\mathbb{Z})$ is injective. So $(k\circ f)^*:H^3(K(\pi,1);\mathbb{Z})\to H^3(M;\mathbb{Z})\cong \mathbb{Z}$ is nontrivial, and the image of $k^*:H^3(K(\pi,1);\mathbb{Z})\to H^3(N;\mathbb{Z})$ contains an infinite order element. Take $\alpha \in H^3(K(\pi,1);\mathbb{Z})$ such that $\alpha'=k^*(\alpha)\in H^3(N;\mathbb{Z})$ has infinite order. 

By Poincare duality, there exists $\beta'\in H^1(N;\mathbb{Z})$, such that $\alpha'\cup \beta'\in H^4(N;\mathbb{Z})$ is nontrivial. Since $k:N\to K(\pi,1)$ induces an isomorphism on $\pi_1$, $k^*:H^1(K(\pi,1);\mathbb{Z})\to H^1(N;\mathbb{Z})$ is an isomorphism. So there exists $\beta\in H^1(K(\pi,1);\mathbb{Z})$ such that $k^*(\beta)=\beta'$.

By Proposition \ref{kpi1} (2) and the universal coefficient theorem, $H^4(K(\pi,1);\mathbb{Z})$ consists of torsion elements, thus $\alpha\cup \beta \in H^4(K(\pi,1);\mathbb{Z})$ is a torsion element. However, we have 
$$k^*(\alpha\cup \beta)=k^*(\alpha)\cup k^*(\beta)=\alpha'\cup \beta'\ne 0\in H^4(N;\mathbb{Z})\cong \mathbb{Z}$$ 
is not a torsion element. So we get a contradiction, thus such a map $f:M\to N$ does not exist.
\end{proof}




\section{Maps from 3-manifolds to 4-manifolds and $\mathbb{Z}_p$-coefficient (co)homology of $3$-manifold groups}

In this section, we will prove the second case of Theorem \ref{3-4map}, which concerns connected sums of spherical $3$-manifolds and $S^1\times S^2$. So we assume that $M=(\#_{j=1}^n N_j)\#(\#^k S^1\times S^2)$, such that $n\geq 1$ and $1<|\pi_1(N_j)|<\infty$ for all $j$.

We will first consider the special case that there exists a prime number $p$, such that $\pi_1(N_j)\cong \mathbb{Z}_p$ for all $j$. The proof divides to two cases depending on the parity of $p$.

The following lemma is parallel to Proposition \ref{kpi1} (1), with $\mathbb{Z}$-coefficient replaced by $\mathbb{Z}_p$-coefficient, and its proof is also similar.

\begin{lemma}\label{cohomologyKpi1}
Let $M$ be a closed orientable $3$-manifold with prime decomposition 
$$M=(\#_{j=1}^n N_j)\#(\#^k S^1\times S^2)$$
with $n\geq 1$, and there exists a prime number $p$ such that $\pi_1(N_j)\cong \mathbb{Z}_p$ for all $j$. Let $i:M\to K(\pi,1)$ be the map that induces the identity homomorphism on $\pi_1$, then the homomorphism $$i_*:H_3(M;\mathbb{Z}_p)\to H_3(K(\pi,1);\mathbb{Z}_p)$$ is injective.
\end{lemma}

We first consider the case that $p$ is odd.

\begin{proposition}\label{podd}
Let $M$ be a closed orientable $3$-manifold with prime decomposition 
$$M=(\#_{j=1}^n N_j)\#(\#^kS^1\times S^2)$$ 
and $n\geq 1$,  and there is an odd prime $p$ such that $\pi_1(N_j)\cong \mathbb{Z}_p$ for all $j$. Then there is no map $f:M\to N$ to a closed orientable $4$-manifold $N$ that induces an isomorphism on $\pi_1$.
\end{proposition}

\begin{proof}
Suppose there is a map $f:M\to N$ to a closed orientable $4$-manifold that induces an isomorphism on $\pi_1$. We take a map $k:N\to K(\pi,1)$ such that $$k\circ f:M\to N\to K(\pi,1)$$ induces the identity homomorphism on $\pi_1$.

Similar to the second and third paragraphs of the proof of Proposition \ref{aspherical}, with $\mathbb{Z}$-coefficient replaced by $\mathbb{Z}_p$-coefficient (invoking Lemma \ref{cohomologyKpi1}), there are $\alpha\in H^3(K(\pi,1);\mathbb{Z}_p), \beta \in H^1(K(\pi,1);\mathbb{Z}_p)$ such that $k^*(\alpha\cup \beta)\ne 0\in H^4(N;\mathbb{Z}_p)\cong \mathbb{Z}_p$.



By \cite[Chapter XII Section 7]{CE}, $H^*(\mathbb{Z}_p;\mathbb{Z}_p)\cong \mathbb{Z}_p[x,y]/(x^2)$, with $|x|=1,|y|=2$, so the cup product $H^3(\mathbb{Z}_p;\mathbb{Z}_p)\times H^1(\mathbb{Z}_p;\mathbb{Z}_p)\to H^4(\mathbb{Z}_p;\mathbb{Z}_p)$ is trivial. Since $H^*(\mathbb{Z};\mathbb{Z}_p)=\mathbb{Z}_p[t]/(t^2)$ with $|t|=1$, the cup product $H^3(\mathbb{Z};\mathbb{Z}_p)\times H^1(\mathbb{Z};\mathbb{Z}_p)\to H^4(\mathbb{Z};\mathbb{Z}_p)$ is trivial. Since $H^*(K(\pi,1);\mathbb{Z}_p)\cong (\oplus_{j=1}^n H^*(\mathbb{Z}_p;\mathbb{Z}_p))\oplus(\oplus^k{H^*(\mathbb{Z};\mathbb{Z}}_p))$, the restriction of the cup product of $H^*(K(\pi,1);\mathbb{Z}_p)$ on $H^3\times H^1\to H^4$ is trivial, so $\alpha\cup \beta=0$ holds.
Now we get a contradiction, thus there is no map $f:M\to N$ that induces an isomorphism on $\pi_1$.
\end{proof}

The second case is $p=2$. Note that any $3$-manifold with $\pi_1\cong \mathbb{Z}_2$ is homeomorphic to $RP^3$.

\begin{proposition}\label{RP3andZ}
Let $M$ be a closed orientable $3$-manifold with prime decomposition $$M=(\#_{j=1}^n RP_j^3)\#(\#^k S^1\times S^2)$$ and $n\geq 1$. Then there does not exist a map $f:M\to N$ to a closed orientable $4$-manifold $N$ that induces an isomorphism on their fundamental groups.
\end{proposition}

\begin{proof}
Let $\pi=\pi_1(M)$. Suppose there is a map $f:M\to N$ that induces an isomorphism on $\pi_1$, then there is a map $k:N\to K(\pi,1)$ such that $k\circ f:M\to K(\pi,1)$ induces the identity homomorphism on $\pi_1$, and a model of $K(\pi,1)$ is $(\vee_{j=1}^nRP_j^{\infty})\vee (\vee^k S^1)$.

Again, similar to the second and third paragraphs of the proof of Proposition \ref{aspherical}, there are $\alpha\in H^3(K(\pi,1);\mathbb{Z}_2), \beta \in H^1(K(\pi,1);\mathbb{Z}_2)$ such that $k^*(\alpha\cup \beta)\ne 0\in H^4(N;\mathbb{Z}_2)\cong \mathbb{Z}_2$. 
Since $H^3(K(\pi,1);\mathbb{Z}_2)\cong \oplus_{j=1}^nH^3(RP_j^{\infty};\mathbb{Z}_2)$,  $H^1(K(\pi,1);\mathbb{Z}_2)\cong (\oplus_{j=1}^nH^1(RP_j^{\infty};\mathbb{Z}_2))\oplus(\oplus^kH^1(S^1;\mathbb{Z}_2))$, and the cup product preserves these components, we can assume that $\alpha\in H^3(RP_1^{\infty};\mathbb{Z}_2)$ and $\beta\in H^1(RP_1^{\infty};\mathbb{Z}_2)$, up to permuting indices.

Let $p:K(\pi,1)=(\vee_{i=1}^nRP_i^{\infty})\vee (\vee^k S^1)\to RP_1^{\infty}$ be identity on $RP_1^{\infty}$ and be constant on other components. Then  the map $p\circ k:N\to RP_1^{\infty}$ satisfies $(p\circ k)^*(\alpha\cup \beta)\ne 0 \in H^4(N;\mathbb{Z}_2)$. So $(p\circ k)^*:H^4(RP_1^{\infty};\mathbb{Z}_2)\cong \mathbb{Z}_2\to H^4(N;\mathbb{Z}_2)\cong \mathbb{Z}_2$ is an isomorphism. 

The homomorphism $\mathbb{Z}\to \mathbb{Z}_2$ on coefficients induces the following commutative diagram
\begin{diagram}
\mathbb{Z}_2 \cong & H^4(RP_1^{\infty};\mathbb{Z}) & \rTo{(p\circ k)^*} & H^4(N;\mathbb{Z}) & \cong  \mathbb{Z}\\
&\dTo                    &          &  \dTo &   \\
\mathbb{Z}_2 \cong & H^4(RP_1^{\infty};\mathbb{Z}_2) & \rTo{(p\circ k)^*} & H^4(N;\mathbb{Z}_2) & \ \cong  \mathbb{Z}_2.\\
\end{diagram}
We know that the first vertical homomorphism and the second horizontal homomorphism are isomorphisms, but a self-isomorphism of $\mathbb{Z}_2$ can not factor through $\mathbb{Z}$ in the upper right corner. So we get a contradiction.

\end{proof}

Now we are ready to prove the second case of Theorem \ref{3-4map}.

\begin{proposition}\label{2powerandZ}
Let $M$ be a closed orientable $3$-manifold with prime decomposition 
$$M=(\#_{j=1}^n N_j)\#(\#^k S^1\times S^2)$$ 
and $n\geq 1$, such that each $\pi_1(N_j)$ is a nontrivial finite group. Then there does not exist a map $f:M\to N$ to a closed orientable $4$-manifold $N$ that induces an isomorphism on their fundamental groups.
\end{proposition}

\begin{proof}
Suppose that there exists a map $f:M\to N$ that induces an isomorphism on $\pi_1$. 

We have a surjective homomorphism 
$$\rho:\pi_1(M)\cong (*_{j=1}^n\pi_1(N_j))*(*^k \mathbb{Z})\to \oplus_{j=1}^n\pi_1(N_j).$$
Here $\rho$ maps each free factor $\pi_1(N_j)$ to the direct summand $\pi_1(N_j)$ by identity, and maps each free factor $\mathbb{Z}$ to the unit element. 
We take a cyclic subgroup $C<\pi_1(N_1)$ with prime order $p$ and denote $\bar{C}=C\times \{1\}\times \cdots \times \{1\}<\oplus_{j=1}^n\pi_1(N_j)$.

Then $K=\rho^{-1}(\bar{C})<\pi_1(M)$ is a finite index subgroup. By the Kurosh Theorem in group theory (\cite{ScW}), we have
$$K\cong (*_{l=1}^a C_l)*(*^b \mathbb{Z}),$$
where each  $C_l$ is a conjugation of $C<\pi_1(N_1)<\pi_1(M)$ in $\pi_1(M)$, and
$a\ge 1$ holds.

Let $\tilde{M}$ and $\tilde{N}$  be the finite covering spaces of $M$ and $N$ corresponding to $K<\pi_1(M)\cong \pi_1(N)$ respectively.
Then the map $f:M\to N$ lifts to a map $\tilde{f}:\tilde{M}\to \tilde{N}$, and it induces an isomorphism on $\pi_1$. 
Here we have  $$\tilde{M}= (\#_{l=1}^a \tilde{N}_l)\#(\#^b S^1\times S^2).$$ 
with $a\geq 1$ and $\pi_1(\tilde{N}_l)\cong C_l\cong \mathbb{Z}_p$ for all $l$.  
The existence of $\tilde{f}$ contradicts Propositions \ref{podd} and \ref{RP3andZ}, thus such a map $f:M\to N$ does not exist.
\end{proof}

\begin{remark} In the proof of Proposition \ref{2powerandZ},  if we denote $n_j=|\pi_1(N_j)|$, then we have 
$$a=\prod_{j=2}^nn_j,\ b=(k+n-1-\frac{p}{n_1}-\sum_{j=2}^n\frac{1}{n_j})\cdot \frac{\prod_{j=1}^nn_j}{p}+1.$$
\end{remark}

\section{Embedding $n$-manifolds into $(n+1)$-manifolds}

\begin{proof}[Proof of Theorem \ref{n}]
Suppose there is an embedding $i: M\to N$ such that $i_* : \pi_1(M) \to \pi_1(N)$ is an isomorphism. We can assume that $i(M)$ is contained in the interior of $N$ (by pushing $M$ into the interior of $N$).
We also suppose that $M=U\#V$, where $U$ is aspherical.

Note $\pi_1(U\#V)=\pi_1(U)*\pi_1(V)$.
We construct a homomorphism
$$\phi=p\circ (i_*)^{-1} : \pi_1(N) \xrightarrow{(i_*)^{-1}} \pi_1(M)=\pi_1(U)*\pi_1(V)\xrightarrow{p} \pi_1(U),$$
where $p: \pi_1(U)*\pi_1(V)\to \pi_1(U)$ maps $\pi_1(V)$ to the unit, and is the identity on
$\pi_1(U)$.

Since $U$ is aspherical, there is a map $f: N\to U$ such that 
$$f_*=\phi : \pi_1(N) \to  \pi_1(U).$$
Let $$g=f\circ i : M= U\# V\to U.$$
Let  $q:  U\# V\to U$ be the map that pinches $V$ to a point.
Clearly $q$ is a map of degree one. Since $g$ and $q$ induce the same homomorphism on $\pi_1$ and $U$ is aspherical, 
$g$ and $q$ are homotopic to each other, thus $g$ is a map of degree one.

Since both $M$ and $N$ are orientable and $i:M\to N$ is a codimension-$1$ embedding, $M$ is two-sided in $N$. 
Note that  $M$ must separate $N$. Otherwise there is a loop $\gamma$ in $N$  which meets $M$ transversely at one point,
thus the  algebraic intersection number of $\gamma$ and $M$ is one. Then by a basic fact in algebraic topology,
$\gamma$ can not be homotoped into $M$, which contradicts the assumption that $i_* : \pi_1(M) \to \pi_1(N)$ is an isomorphism.

Since $M$ separates $N$ and $\partial N$ has at most one component, the map $i: M\to N$ induces a null homomorphism on
$i_* :  H_n(M, \mathbb{Z})\to H_n(N, \mathbb{Z})$.
It follows that 
$$g_*= f_*\circ i_*: H_n(M,\mathbb{Z})\to H_n(N,\mathbb{Z})\to H_n(U,\mathbb{Z})$$
is a null homomorphism. Hence $g$ cannot have degree 1, and we get a contradiction.
\end{proof}

\section{Maps from 4-manifolds to 3-manifolds}

Although maps from closed orientable 3-manifolds to closed orientable 4-manifolds realizing isomorphisms on $\pi_1$ are rare (Theorem \ref{3-4map}),
there are plenty of maps from closed orientable 4-manifolds to closed orientable 3-manifolds realizing isomorphisms on $\pi_1$, as constructed below.

\begin{example}\label{4-3}
For any closed orientable 3-manifold $M$, we can construct  a closed orientable 4-manifold $M^*$ with the same $\pi_1$, by taking
$$M^*=((M\setminus D^3)\times S^1) \cup (S^2\times D^2).$$
Here $D^3\subset M$ is a 3-ball, then $\partial ((M\setminus D^3)\times S^1)=S^1\times S^2$ and $\partial (S^2\times D^2)=S^1\times S^2$ are identified canonically.
Clearly $$\pi_1(M^*)\cong \pi_1(M\setminus D^3)\cong\pi_1(M).$$

Now we construct  a map $f:M^*\to M$ that induces an isomorphism on $\pi_1$, which is a composition $f=f_3\circ f_2 \circ f_1$ as below.
Fix $y\in S^1$, then 
$$f_1: M^*=((M\setminus D^3)\times S^1) \cup (S^2\times D^2)\to ((M\setminus D^3)\times \{y\}) \cup (S^2\times S^2)$$
 is a quotient  map that projects $((M\setminus D^3) \times S^1, \partial D^3\times S^1)$ to $((M\setminus D^3)\times \{y\}, \partial D^3\times \{y\})$,  
and maps $(S^2\times D^2, S^2\times \partial D^2)$ to $(S^2\times S^2, S^2\times \{y\})$, and  $\partial D^3\times \{y\}$ is identified with
 $S^2\times \{y\}$.
$$f_2:((M\setminus D^3)\times \{y\}) \cup (S^2\times S^2)\to (M\setminus D^3)\times \{y\}$$
is the quotient map that is the identity on $(M\setminus D^3)\times \{y\}$ and projects $S^2\times S^2$ to $S^2\times \{y\}=\partial D^3\times \{y\}$.
$$f_3: (M\setminus D^3)\times \{y\}  \to M$$
is the inclusion.
\end{example}


For each closed orientable 3-manifold $M$,  $\chi_4(\pi_1(M))$, the minimum Euler characteristic among all closed orientable 4-manifolds having $\pi_1$ isomorphic to $\pi_1(M)$, was studied, 
and the $M^*$ constructed in Example \ref{4-3}  is  a closed orientable 4-manifold realizing $\chi_4(\pi_1(M))$ when $M$
contains no $S^1\times S^2$ prime factor, see
\cite{HW},  \cite{Ko}, \cite{Hi}, \cite{KL}, \cite{SW}.

\begin{question}\label{q4-3} 
Suppose a closed orientable 4-manifold $N$ and a closed orientable 3-manifold $M$ have isomorphic $\pi_1$.

(1) Is there a  map $f:N\to M$ inducing an isomorphism
on their $\pi_1$?

(2) Is there a degree-$1$ map $f:N\to X$ to some closed orientable 4-manifold $X$ realizing $\chi_4(\pi_1(M))$? 
In particular, is there a degree-$1$ map $f:N\to M^*$?
\end{question}

\begin{remark}


(i)  Lens spaces $L(5,1)$ and $L(5,2)$ are closed orientable 3-manifolds having isomorphic $\pi_1$ but are not homotopy equivalent.
However, since there is a homotopy equivalence $h:  L(5,1)\setminus D^3 \to L(5,2)\setminus D^3$, 
the  map $g=i\circ h\circ f_2\circ f_1: L(5,1)^*\to L(5,2)$ induces an isomorphism on $\pi_1$. Here $f_2\circ f_1: L(5,1)^*\to L(5,1)\setminus D^3$ is constructed in Example \ref{4-3} and $i: L(5,2)\setminus D^3\to L(5,2)$ is the inclusion.


(ii) If the last sentence of Question \ref{q4-3} (2) has a positive answer for some $M$, then Question \ref{q4-3} (1) has a positive answer for the same $M$. Because each degree-one map  induces a surjection on $\pi_1$ and $\pi_1(M)$ is co-hopfian.


(iii) On the other hand, one may ask Question \ref{q4-3} (2) for any finitely presented group $G$: If $N$  is a closed orientable 4-manifold
with $\pi_1(N)\cong G$,  does $N$ 1-dominate a 4-manifold $X$ realizing $\chi_4(G)$?  
\end{remark}

\end{document}